\RequirePackage{amsthm}
\documentclass[sn-mathphys-num]{sn-jnl}

\usepackage{graphicx}%
\usepackage{multirow}%
\usepackage{amsmath,amssymb,amsfonts}%
\usepackage{amsthm}%
\usepackage{mathrsfs}%
\usepackage[title]{appendix}%
\usepackage{xcolor}%
\usepackage{textcomp}%
\usepackage{manyfoot}%
\usepackage{booktabs}%
\usepackage{algorithm}%
\usepackage{algorithmicx}%
\usepackage{algpseudocode}%
\usepackage{listings}%

\usepackage{hyperref}
\usepackage{cleveref}

\theoremstyle{thmstyleone}
\newtheorem{theorem}{Theorem}[section]
\newtheorem{proposition}[theorem]{Proposition}

\newtheorem{lemma}[theorem]{Lemma}
\newtheorem{corollary}[theorem]{Corollary}
\newtheorem{notation}[theorem]{Notation}
\theoremstyle{thmstyletwo}
\newtheorem{remark}[theorem]{Remark}
\theoremstyle{thmstylethree}
\newtheorem{definition}[theorem]{Definition}

\newcommand{\Lip}{\operatorname{Lip}}

\newcommand{\dist}{\operatorname{dist}}

\newcommand{\om}{\Omega}
\newcommand{\scut}{\om^*}
\newcommand{\cubes}{\mathcal C}
\newcommand{\cube}{Q}
\newcommand{\alm}{m}
\newcommand{\A}{\mathcal{A}_\alm}
\newcommand{\B}{\mathcal{B}_\alm}

\renewcommand{\d}{\, \mathrm{d}}

\newcommand{\M}{\mathcal M}
\renewcommand{\P}{\mathcal P}
\newcommand{\Wass}{W}
\newcommand{\bWass}{Wb}
\newcommand{\bM}{\M b}

\newcommand{\ldir}{[\![}
\newcommand{\rdir}{]\!]}

\renewcommand{\L}{\phi}
\newcommand{\hLi}{\phi^*}
\newcommand{\lsum}{\mathcal T}

\newcommand{\N}{\mathbb N}
\newcommand{\R}{\mathbb R}

\begin{document}
	
	\title[Bi-Lipschitz embeddings and partial transportation]{Bi-Lipschitz embeddings of the space of unordered $\alm$-tuples with a partial transportation metric}

	\author*[1]{\fnm{David} \sur{Bate}}
	\author[2]{\fnm{Ana Luc\'ia} \sur{Garc\'ia Pulido}}

	\affil*[1]{Mathematics Institute, University of Warwick, Coventry, CV4 7AL.\\
	ORCiD: \href{https://orcid.org/0000-0003-0808-2453}{0000-0003-0808-2453}, \href{mailto:david.bate@warwick.ac.uk}{david.bate@warwick.ac.uk}}

	\affil[2]{
		Division of Computing Science and Mathematics,\\
		University of Stirling, Stirling, FK9 4LA.\\
	ORCiD: \href{https://orcid.org/0000-0002-4163-0768}{0000-0002-4163-0768}, \href{mailto:analucia.garciapulido@stir.ac.uk}{analucia.garciapulido@stir.ac.uk}}

	\abstract{
		Let $\om\subset \R^n$ be non-empty, open and proper.
		This paper is concerned with $\bWass_p(\om)$, the space of $p$-integrable Borel measures on $\om$ equipped with the \emph{partial} transportation metric introduced by Figalli and Gigli that allows the creation and destruction of mass on $\partial \om$.
		Alternatively, we show that $\bWass_p(\om)$ is isometric to a subset of Borel measures with the ordinary Wasserstein distance, on the one point completion of $\om$ equipped with the shortcut metric
	\[\delta(x,y)= \min\{\|x-y\|, \dist(x,\partial \om)+\dist(y,\partial\om)\}.\]

	In this article we construct bi-Lipschitz embeddings of the set of unordered $m$-tuples in $\bWass_p(\om)$ into Hilbert space.
	This generalises Almgren's bi-Lipschitz embedding theorem to the setting of optimal partial transport.
	}
		
	\maketitle

	\bmhead{Acknowledgements}
	D.B. was supported by the European Union's Horizon 2020 research and innovation programme grant number 948021. A.L.G.P. was supported by the Engineering and Physical Sciences Research Council grant number EP/R018472/1.

	We would like to thank Andrea Marchese for useful discussions regarding Almgren's $\alm$-valued functions.
	We would also like to thank the referee for carefully reading this article and providing valuable suggestions that improved the exposition of this work.
	
	\section{Introduction}

	A striking variety of problems in geometry, analysis, combinatorics and a vast number of applications can be neatly formulated in terms of measures and their comparison using transportation metrics.
	The prototypical transportation metric is the \emph{$p$-Wasserstein distance} \cite{abs}.
	This is defined between two Borel measures of the same total mass on a metric space $(X,d)$ by
	\begin{equation}
		\label{wass}
		\Wass_p(\mu,\nu) = \inf_{\gamma} \left(\int_{X\times X} d(x,y)^p \d\gamma(x,y)\right)^{\frac{1}{p}},
	\end{equation}
	where $p\geq 1$ and the infimum is taken over all measures $\gamma$ on $X\times X$ with coordinate projections $\pi_1\gamma=\mu$ and $\pi_2\gamma=\nu$.
	The resulting metric space of $p$-integrable probability measures equipped with $\Wass_p$ is denoted by $\Wass_p(X)$ (see \Cref{def-wass}).
	
	A drawback of the Wasserstein distance is the requirement that the compared measures must have the same total mass.
	Recently emerging theories of \emph{optimal partial transport} pertain to the transportation of measures without a mass constraint \cite{zbMATH06845438,zbMATH06856660,zbMATH06260964}.
	This article concerns the following formulation due to Figalli and Gigli \cite{zbMATH05771280}.
	
	Let $\om$ be an open non-empty proper subset of $X$.
	For \emph{measures} $\mu$ and $\nu$ on $\om$, one defines $\bWass_p(\mu,\nu)$ as in \eqref{wass}, but the infimum is taken over measures $\gamma$ on $\overline \om \times \overline \om$ with
	\[\pi_1 \gamma|_{\om}=\mu \quad \text{and} \quad \pi_2 \gamma|_{\om}=\nu.\]
	The resulting metric space of $p$-integrable measures, equipped with $\bWass_p$, will be denoted by  $\bWass_p(X)$ (see \Cref{FG-wass}).

	The key property of $\bWass_p$ is that $\partial\om$ can be used to destroy or create mass, at a cost of transporting it to or from $\partial \om$.
	This allows measures of different total masses to be compared and hence one can construct a metric space consisting of \emph{all} measures, instead of restricting to probability measures.
	Understanding the interplay between transportation metrics and $\partial \om$ is motivated by solving evolution equations with Dirichlet boundary conditions from gradient flows \cite{zbMATH05771280,zbMATH07216174}.
	The metric $\bWass_p$ has found further applications such as obtaining new comparison principles for viscosity solutions \cite{MR4024555}.
		
	A natural approach to study a metric space is to embed it into a well known space, such as a Euclidean or Banach space, as this allows the metric space to inherit geometric properties of the ambient space.
	Recall that the \emph{distortion} of an injective map $f$ between two metric spaces is $\Lip(f)\cdot\Lip(f^{-1})$, where $\Lip(f)$ is the Lipschitz constant of $f$;
	$f$ is \emph{bi-Lipschitz} if it has finite distortion.
	Since bi-Lipschitz embeddings preserve relative distances, they are central to analysis and metric geometry \cite{naorICM} and have applications to algorithm design \cite{indyk}.

	Due to the prominence of the Wasserstein spaces in various areas of mathematics, their embeddability has attracted much attention.
	The non-embeddability (into $L^1$) of $\Wass_1$ over various discrete metric spaces \cite{BGS} such as the planar grid \cite{NS} and Hamming cube \cite{KN} is known, as is the non-embeddability of $\Wass_p(\R^3)$ for $p\geq 1$ \cite{naor}.
	The interest in bi-Lipschitz embeddings of the Wasserstein spaces dates back to the work of Almgren \cite{zbMATH01528183,MR2663735},
	forming the foundations of his celebrated partial regularity theorem for area minimising currents.
	Almgren proved that, for any $\alm\in \N$, the set of unordered $\alm$-tuples of points in $\R^n$,
	\[\A(\R^n) = \left\{\sum_{i=1}^\alm \ldir x_i\rdir  : x_i\in \R^n\ \forall 1\leq i \leq \alm\right\}\]
	equipped with $\Wass_2$, bi-Lipschitz embeds into some Euclidean space (see \Cref{almgren}).
	Here and throughout, $\ldir x \rdir$ will denote the Dirac mass at $x$.

	In this article we generalise Almgren's embedding to $\bWass_2(\om)$.
	\begin{theorem}\label{main-intro}
		For $n\in\N$, let $\om\subset \R^n$ be non-empty, open and proper.
			The space $(\B(\om),\bWass_2)$ of unordered tuples of at most $\alm$ points bi-Lipschitz embeds into Hilbert space.
			The distortion of our embedding is at most $c\alm^{n+5/2}$, for some constant $c\geq 1$.
	\end{theorem}
	In general, $\bWass_p(\om)$ is not a doubling metric space and hence cannot be bi-Lipschitz embedded into any Euclidean space, see \Cref{non-doubling}.
	Therefore Hilbert space\footnote{We adopt the standard convention that Hilbert space is the unique complete and separable infinite dimensional inner product space, up to isometric isomorphism.} becomes the natural target for an embedding.
	Note that, since we are not constrained to comparing measures of the same total mass, in \Cref{main-intro} we consider unordered tuples of \emph{at most} $\alm$-points.

	To prove \Cref{main-intro}, we first show, for $\om\subset X$, that $\bWass_p(\om)$ isometrically embeds into the ordinary $p$-Wasserstein space of measures on $(\scut,\delta)$, where $\scut$ is the one point completion of $\om$ equipped with the shortcut metric 
	\[
		\delta(x,y)= \min\{\|x-y\|, \dist(x,\partial \om)+\dist(y,\partial \om)\}
	\] for every $x,y\in\om$ (see \Cref{scut-equiv}).
	This embedding maps $\B(\om)$ to $\A(\scut)$ and so, in order to prove \Cref{main-intro}, it remains to construct a bi-Lipschitz embedding of $\A(\scut)$ into Hilbert space. 

	We do this, for $\om\subset \R^n$, by considering a Whitney decomposition $\cubes$ of $\om$ into cubes.
	This decomposition is chosen such that, inside any cube $\cube\in\cubes$, the shortcut metric equals the Euclidean metric and consequently
	\begin{equation}
		\label{intro-Q-equal}
		\A(\cube,\delta) = \A(\cube,\|\ \|).
	\end{equation}
	In particular, Almgren's theorem gives an embedding of each $\A(\cube)$ into some Euclidean space.
	Despite the fact that any measure can be written as a sum of measures supported on cubes in $\cubes$, the construction of the required bi-Lipschitz embedding of $\A(\scut)$ cannot be obtained simply by restricting to cubes.
	Indeed, $\Wass_p$ may not even be defined between the restriction of two measures to a cube;
	even when it is, simple examples show that the optimal transport of the restricted measures may be incomparable to the optimal transport of the original measures.

	Our approach uses \eqref{intro-Q-equal} as the starting point to determine the optimal transport of measures between different cubes, see \Cref{coord-proj}.
	From this analysis we construct a bi-Lipschitz embedding of $\A(\scut)$ into the $\ell_2$-sum of infinitely many copies of $\A(\R^{n+1})$, see \Cref{embed-into-wass}.
	The proof of \Cref{main-intro} is concluded in \Cref{sec4} by applying Almgren's embedding to each term of the $\ell_2$-sum.

	We mention an application of \Cref{main-intro} to persistence homology.
	The space of \emph{persistence barcodes} can be viewed as $\cup_m\B(U)$ for
		\[
		U	= \{ (x,y)\in\R^2: y> x\},
		\] 
		see \cite{zbMATH07421454}. 
		\Cref{main-intro} shows that the space of persistence barcodes with at most $\alm$-points can be bi-Lipschitz embedded into Hilbert space.
		This answers questions raised by 
		Carri{\`e}re and Bauer \cite{bauer}.
		Prior to our results, it was known that $\B(U)$ coarsely embeds into Hilbert space \cite{mitravirk}. 
		In fact, \Cref{main-intro} applies to the generalised persistence barcodes introduced in \cite{2109.14697} whenever the ambient space is Euclidean.
		Our theorem also holds when $\B(\om)$ is equipped with any $\Wass_p$ for $p\geq 1$; due to the equivalence of norms on $\R^\alm$, these metrics are all bi-Lipschitz equivalent.

	Finally, we mention that the distortion of any embedding of $\B(\om)$ into Hilbert space, for $\om\subset \R^n$, must necessarily converge to $\infty$ as $\alm$ does, see \Cref{distortion}. 

	\section{Wasserstein distance and Almgren's embedding}\label{sec-wass}

	Let $(X,d)$ be a complete and separable metric space.
	We write $\M(X)$ for the set of Borel measures on $X$ and $\P(X)$ for the set of Borel probability measures on $X$.
	The Wasserstein space is defined as follows \cite{abs,ags}.
	\begin{definition}
		\label{def-wass}
		For $\mu,\nu\in \M(X)$ and $p \in [1,\infty)$ define
		\[\Wass_p(\mu,\nu) = \inf_{\gamma} \left(\int_{X\times X} d(x,y)^p \d\gamma(x,y)\right)^{\frac{1}{p}},\]
		where the infimum is taken over all \emph{couplings} $\gamma \in \M(X\times X)$ with coordinate projections $\pi_1\gamma=\mu$ and $\pi_2\gamma=\nu$.
		Note that $\Wass_p(\mu,\nu)<\infty$ only if $\mu(X)=\nu(X)$ as otherwise there does not exist a $\gamma$ as in \Cref{def-wass}.
	
		Let $\P_p(X)$ be those $\mu\in \P(X)$ with
		\[\int_{X} d(x,x_0)^p \d\mu(x) <\infty\]
		for some (equivalently all) $x_0\in X$.
		Then $\Wass_p$ defines a metric on $\P_p(X)$.
		Analogous statements hold for the case $p=\infty$, where the $L^p$ integral is replaced by an essential supremum.
		We write $\Wass_p(X)$ for the set $\P_p(X)$ equipped with $\Wass_p$.
	\end{definition}

	\begin{definition}\label{tuples}
		For $\alm \in \N$, define the space of \emph{unordered $m$-tupes}
		\[\A(X) = \left\{\sum_{i=1}^\alm \ldir x_i\rdir  : x_i\in X\ \forall 1\leq i \leq \alm\right\},\]
		equipped with $\Wass_2$.
		Note that, on $\A(X)$, $\Wass_2$ equals
		\begin{equation*}
			\Wass_2(p,q) = \min_{\sigma\in \Sigma_\alm} \sqrt{\sum_{i=1}^\alm d(p_i,q_{\sigma(i)})^2},
		\end{equation*}
		where $p=\sum_{i=1}^\alm \ldir p_i\rdir $ and $q=\sum_{i=1}^\alm \ldir q_i\rdir $.
	\end{definition}

	A fundamental step in Almgren's study of area minimising currents was the following bi-Lipschitz embedding.
	\begin{theorem}[Almgren, Theorem 2.1 \cite{MR2663735}]
		\label{almgren}
		For every $\alm\in\N$ there exists an $N\in\N$ and a bi-Lipschitz embedding $\xi\colon \A(\R^n) \to \R^N$.
		By inspecting the proof one sees that $\xi(0)=0$ and, for all $p,q \in \A(\R^n)$,
		\begin{equation*}
			\frac{\Wass_2(p,q)}{c\alm^{n+1}} \leq \|\xi(p)-\xi(q)\| \leq \Wass_2(p,q) 
		\end{equation*}
		for a constant $c\geq 1$.
	\end{theorem}

	\section{Optimal Partial Transport and the Shortcut Metric}
	\label{prelim}
	
	The transportation metric $\bWass$ introduced by Figalli and Gigli \cite{zbMATH05771280} is defined between two Borel measures.
	Originally defined for open and bounded $\om\subset \R^n$, we state the natural generalisation of $\bWass$ to complete and separable metric spaces $(X,d)$ (the proof of the triangle inequality is identical).

	\begin{definition}
		\label{FG-wass}
		Let $\om\subset X$ be proper and non-empty.
		For $\mu,\nu\in \M(\om)$ and $p \in [1,\infty)$ define
		\[\bWass_p(\mu,\nu) = \inf_{\gamma} \left(\int_{X\times X} d(x,y)^p \d\gamma(x,y)\right)^{\frac{1}{p}},\]
		where the infimum is taken over all \emph{couplings} $\gamma \in \M(X\times X)$ with $\pi_1\gamma|_\om=\mu$ and $\pi_2\gamma|_\om=\nu$.
		Then $\bWass_p$ defines a metric on
			\[\bM_p(\om):=\{\mu\in \M(\om) :\bWass_p(\mu,0)<\infty\}.\]
		Analogous statements hold for the case $p=\infty$, where the $L^p$ integral is replaced by an essential supremum.

		We write $\bWass_p(\om)$ for the set $\bM_p(\om)$ equipped with $\bWass_p$.
		We also write $\bWass^1_p(\om)$ for the set of $\mu\in \bM_p(\om)$ with $\mu(\om)\leq 1$, equipped with $\bWass_p$.
	\end{definition}

	The first step in our proof of \Cref{main-intro} is to show an equivalence between $\bWass^1_p(\om)$ and $\Wass_p(\scut)$, for $\scut$ the \emph{shortcut} metric space, defined as the one point completion of $\om$ via its complement.
	\begin{definition}
		\label{scut}
		For $\om\subset X$ non-empty and proper, let $\scut = \om \cup \{\partial\}$.
		For $x,y\in\om$ define
		\begin{equation*}
			\delta(x,y)= \min\{\|x-y\|, \dist(x,X\setminus \om)+\dist(y,X\setminus \om)\}
		\end{equation*}
		and $\delta(x,\partial) = \dist(x,X\setminus \om)$.
		Then $\delta$ defines a metric on $\scut$.
	\end{definition}

	Profeta and Sturm \cite[Remark 1.9]{zbMATH07216174} mention that $\bWass_1^1(\om)$ isometrically embeds into $\Wass_1(\scut)$, and give an example showing that their embedding is not an isometry for $p>1$.
	We show that there exists an isometric embedding of $\bWass_p^1(\om)$ into $2\Wass_p(\scut)$ for any $p\geq 1$.
	Here we write $2\Wass_p(\scut)$ for the space of measures with total mass equal to 2.
	\begin{lemma}
		\label{scut-equiv}
		Let $X$ be a separable metric space and $\om\subset X$ be non-empty and proper.
		For any $p\geq 1$, the map
		\begin{align*}
			\bWass_p^1(\om) &\to 2\Wass_p(\scut)\\
			\iota(\mu) &= \mu + (2-\mu(\om))\ldir \partial \rdir,
		\end{align*}
		is an isometric embedding.
	\end{lemma}

	\begin{proof}
		Given a coupling for $\mu,\nu$ we use it to construct a coupling for $\iota(\mu),\iota(\nu)$ and vice versa.
		
		First let $\mu,\nu\in \bWass_p^1(\om)$ and suppose that $\gamma\in\M(X\times X)$ is a coupling for $\mu$ and $\nu$ in $\bWass_p(\om)$.
		Let $\pi_\partial(x)=\partial$ for all $x\in X$ and define $\gamma'\in \M(\scut\times\scut)$ as
		\begin{align*}
			\gamma' &= \gamma|_{\om\times\om} + (\pi_\partial \times \operatorname{id})_{\#} \gamma|_{X\setminus \om \times \om} + (\operatorname{id} \times \pi_\partial)_{\#} \gamma|_{ \om \times X\setminus \om}\\
			&\qquad +(2-[\gamma(\om\times\om) +\gamma(X\setminus \om\times \om) + \gamma(\om\times X\setminus \om)])\ldir (\partial,\partial) \rdir.
		\end{align*}
		For notational convenience, we let $\kappa$ denote the coefficient of $\ldir (\partial,\partial) \rdir$ in this expression.
		Then
		\begin{align*}
			\pi_1 \gamma' &= \pi_1 (\gamma|_{\om\times \om}) + \gamma(X\setminus \om \times \om)\ldir \partial \rdir + \pi_1(\gamma|_{\om\times X\setminus \om}) + \kappa \ldir \partial \rdir\\
			&= \pi_1 (\gamma|_{\om\times \om}) + \pi_1(\gamma|_{\om\times X\setminus \om}) + (2-[\gamma(\om\times\om)+\gamma(\om\times X\setminus \om)])\ldir \partial \rdir\\
			&= \pi_1(\gamma|_{\om\times X}) + (2-\gamma(\om\times X))\ldir \partial \rdir\\
			&= \mu + (2-\mu(X))\ldir \partial \rdir = \iota(\mu).
		\end{align*}
		Similarly, by symmetry, $\pi_2\gamma' = \iota(\nu)$.
		Thus $\gamma'$ is a coupling of $\iota(\mu)$ and $\iota(\nu)$ in $\Wass_p(\scut)$.
		Moreover,
		\begin{align}\label{1coup}
			\int \delta(x,y)^p\d\gamma'(x,y) &= \int_{\om\times\om} \delta(x,y)^p \d \gamma(x,y)+ \int_{X\setminus \om\times\om} \delta(\partial,y)^p\d\gamma(x,y)\notag\\
			&\quad + \int_{\om\times X\setminus \om} \delta(x,\partial)^p d\gamma(x,y) +\kappa\delta(\partial,\partial)^p\notag\\
			&\leq  \int_{\om\times\om} d(x,y)^p \d\gamma(x,y) + \int_{X\setminus \om\times\om} d(x,y)^p \d\gamma(x,y)\notag\\
			&\quad+ \int_{\om\times X\setminus \om} d(x,y)^p\d\gamma(x,y)\notag\\
			&=\int d(x,y)^p\d\gamma(x,y).
		\end{align}
		Therefore,
		\[\Wass_p(\iota(\mu),\iota(\nu)) \leq \bWass_p(\mu,\nu).\]
		
		Conversely, let $\gamma$ be a coupling for $\iota(\mu)$ and $\iota(\nu)$ in $\Wass_p(\scut)$.
		Define the closed set
		\[E=\{(x,y)\in \om\times\om: \delta(x,y) = d(x,y)\}.\]
		Fix $\epsilon>0$ and for each $x\in \om$, let $c(x)\in X\setminus \om$ with
		\[d(x,c(x)) \leq (1+\epsilon)\dist(x,X\setminus \om).\]
		Since $X$ is separable, $c$ may be chosen to be a Borel function with countable image.
		Let $c_1=(\operatorname{id}\times c)\circ \pi_1$ and $c_2=(c\times\operatorname{id})\circ \pi_2$ and define
		\begin{equation}\label{gammaprimed}
			\gamma' = \gamma|_E + (c_1)_{\#} \gamma|_{(\om\times\scut)\setminus E} + (c_2)_{\#} \gamma|_{(\scut\times\om)\setminus E} \in \M(X\times X).
		\end{equation}
		Note that, since $\pi_1((c_1)_{\#}\gamma)$ is supported on $X\setminus \om$, its restriction to $\om$ equals 0.
		Therefore,
		\begin{equation*}(\pi_1\gamma')|_\om = (\pi_1\gamma|_E)|_\om + (\pi_1\gamma|_{(\om\times\scut)\setminus E})|_\om + 0 = (\pi_1 \gamma)|_\om = \mu.
		\end{equation*}
		Similarly, by symmetry, $(\pi_2\gamma')|_\om = \nu$.
		Hence $\gamma$ is a coupling for $\mu$ and $\nu$ in $\bWass_p(\om)$.

		Now, for any $(x,y)\in (\om\times\om)\setminus E$,
		\[d(x,c(x))^p + d(c(y),y)^p \leq (1+\epsilon)^p(\delta(x,\partial)^p + \delta(\partial,y)^p) \leq (1+\epsilon)^p\delta(x,y)^p.\]
		Therefore,
		\begin{align}\label{coup2}
			\int_{X\times X} d(x,y)^p\d\gamma'(x,y) &= \int_E d(x,y)^p\d\gamma(x,y) + \int_{(\om\times\scut)\setminus E} d(x,c(x))^p\d\gamma(x,y)\notag\\
			&\quad + \int_{(\scut\times \om)\setminus E} d(c(y),y)^p\d\gamma(x,y)\notag\\
			&\leq \int_E d(x,y)^p\d\gamma(x,y) + \int_{(\scut\times \scut)\setminus E} (1+\epsilon)^p\delta(x,y)^p\d\gamma(x,y)\notag\\
			&\leq (1+\epsilon)^p\int_{\scut\times\scut} \delta(x,y)^p\d\gamma(x,y).
		\end{align}
		Since $\epsilon>0$ is arbitrary, this shows that
		\[\Wass_p(\iota(\mu),\iota(\nu)) \geq \bWass_p(\mu,\nu).\]
	\end{proof}

	\begin{remark}
		After the first version of this article appeared, we were made aware that the statement of \Cref{scut-equiv}, for the case $\om=U$ as defined in our introduction, appears in the work of Divol and Lacombe \cite[Proposition 3.15]{zbMATH07421454}.
		Note that our proof does not rely on the existence of unique closest points in $\partial \om$, whilst the one in \cite{ zbMATH07421454} does.
		However, a flaw in their argument makes the proof incorrect even for the case of $\om=U$.

		Central to their proof is the definition of a measure $\tilde \pi'$ and the claim that it is a coupling of $\tilde\mu$ and $\tilde\nu$ in $\Wass_p(\scut)$ (using the variables of \cite[Lemma 3.17]{zbMATH07421454}).
		Using this they derive \cite[Equation (3.8)]{zbMATH07421454} from which the proof is concluded.
		However, examples such as \cite[Remark 1.9]{zbMATH07216174} show this equation to be false.
		Moreover, this equation would imply that $\delta=d$ in $\om$.
		These contradictions originate in the fact that $\tilde\pi'$ is not a coupling of $\tilde\mu$ and $\tilde \nu$, which can be verified by comparing the total measure of $\tilde\pi'$ to that of $\tilde\mu,\tilde\nu$ or $\tilde\pi$.
	\end{remark}

	Since the map
		\begin{align*}
			2\Wass_p(\scut) &\to \Wass_p(\scut)\\
			\mu &\mapsto \mu/2
		\end{align*}
	has distortion $2^{1/p}$,
	we obtain the following corollary.

	\begin{corollary}
		Let $X$ be a separable metric space and $\om\subset X$ be non-empty and proper.
		Then $\bWass_p^1(\om)$ bi-Lipschitz embeds into $\Wass_p(\scut)$ with distortion $2$.
	\end{corollary}

	The same proof as the one for \Cref{scut-equiv} shows that the full space $\bWass_p(\om)$ isometrically embeds into $\M(\scut)$.
	\begin{lemma}
		\label{inf-boundary}
	Let $X$ be a separable metric space and $\om\subset X$ non-empty and proper.
	For any $p\geq 1$,
		\begin{align*}
			\bWass_p(\om) &\to (\M(\scut),\Wass_p)\\
			\iota'(\mu) &= \mu + \infty \cdot \ldir \partial \rdir
		\end{align*}
		is an isometric embedding.
	\end{lemma}
	\begin{remark}
	For any $\mu\in \bM_p(\om)$,
		\[\Wass_p(\iota'(\mu),\infty\cdot\ldir \partial\rdir)=\bWass_p(\mu,0)<\infty.\]
		Therefore, the triangle inequality for $\Wass_p$ implies that $\Wass_p$ is indeed a metric on the image of $\iota'$.
	\end{remark}

	\begin{proof}[Proof of \Cref{inf-boundary}]
			If $\mu,\nu\in \bM_p(\om)$ then
		\[\gamma' = \gamma|_{\om\times\om} + (\pi_\partial \times \operatorname{id})_{\#} \gamma|_{X\setminus \om \times \om} + (\operatorname{id} \times \pi_\partial)_{\#} \gamma|_{ \om \times X\setminus \om} +\infty \cdot \ldir (\partial,\partial)\rdir\]
		defines a coupling of $\iota'(\mu)$ and $\iota'(\nu)$.
		The calculation in \eqref{1coup} shows that
		\[\Wass_p(\iota'(\mu),\iota'(\nu)) \leq \bWass_p(\mu,\nu).\]

		Conversely, if $\mu,\nu\in \bM_p(\om)$, then $\gamma'$ as defined in \eqref{gammaprimed} is a coupling for $\mu,\nu$ and \eqref{coup2} shows that
		\[\Wass_p(\iota'(\mu),\iota'(\nu)) \geq \bWass_p(\mu,\nu).\]
		\end{proof}
		
	\subsection{The shortcut metric space is not doubling}

	A metric space $X$ is \emph{doubling} if there exists $N\in\N$ such that each ball $B\subset X$ is covered by $N$ balls of half the radius of $B$.
	\begin{lemma}
		\label{non-doubling}
		For $n\geq 2$, let $\om\subset \R^n$ be non-empty and open such that $\overline \om$ is a proper subset of $\R^n$.
		Then for any $N\in\N$ and any sufficiently small $\epsilon>0$, there exist $y_1,\ldots,y_N\in \om$ with $\delta(y_i,y_j)=\epsilon$
		for each $i\neq j$.
		In particular, $\scut$ is not doubling.
	\end{lemma}

	\begin{proof}
		Let $x\not\in \overline \om$ and $y\in \om$.
		For $N\in \N$, let $y_1,\ldots,y_N\in \om$ lie on the circle centred on $x$ of radius $\|x-y\|$ (such points exist since $\om$ is open). 
		For each $1\leq i\leq N$, let $l'_i$ be the line segment connecting $y_i$ to $x$ and let $l_i$ be the connected component of $l'_i\cap \om$ containing $y_i$.
		Since $x\not\in \overline \om$, there exists $\eta>0$ such that
		\[\inf\{ \|z-z'\| : z\in l_i,\ z'\in l_j,\ i\neq j\} >\eta.\]

		Now, $\dist(\cdot,\partial \om)$ is continuous on each $l_i$ and converges to 0 as one travels along $l_i$ towards $\partial \om$.
		Therefore, for each sufficiently small $\epsilon>0$ and each $1\leq i \leq N$, there exists $z_i\in l_i$ with $\dist(z_i,\partial \om)=\epsilon/2$.
		In particular, if $\epsilon<\eta$, then $\delta(z_i,z_j)=\epsilon$ for each $1\leq i\neq j\leq N$.

		Finally, we see that $y_i\in B(y_1,\epsilon)$ for each $1\leq j\leq N$, but we require at least $N$ balls of radius $\epsilon/4$ to cover $B(y_1,\epsilon)$.
		Since $N\in\N$ is arbitrary, $\scut$ cannot be doubling.
	\end{proof}

	\begin{remark}
		\Cref{non-doubling} is sharp in the following sense.
		If $\om=(-1,1)\subset \R$, then $\scut$ is bi-Lipschitz equivalent to a Euclidean circle.
		For any $n\in \N$, if $\om= \R^n\setminus \{0\}$, then $\scut$ is isometric to $\R^n$.
		In both of these cases, the conclusion of \Cref{non-doubling} fails.
	\end{remark}

	Note that each Euclidean space is doubling and that the doubling property is preserved under taking subsets and bi-Lipschitz images.
	Therefore, if a metric space is bi-Lipschitz embeddable into some Euclidean space, it must necessarily be doubling.

	\begin{corollary}
		\label{not-embed-euclidean}
		For $n\geq 2$ let $\om\subset \R^n$ be non-empty and open such that $\overline \om$ is a proper subset of $\R^n$.
		Then $\scut$ is not bi-Lipschitz embeddable into any Euclidean space.
	\end{corollary}

	\subsection{The space of unordered tuples of at most \texorpdfstring{$\alm$}{m} points}

	\begin{definition}
		\label{atmost-tuples}
		Let $X$ be a metric space, $\om\subset X$ non-empty and proper and $\alm\in\N$.
		Define the \emph{space of unordered tuples of at most $m$ points} as
		\[
			\B(\om) = \bigcup_{k=1}^\alm \mathcal{A}_k(\om),
		\]
		with the metric inherited from $\bWass_2(\om)$.
	\end{definition}

	This space is naturally identified with a subset of $\A(\scut)$.
	\begin{corollary}
		\label{B-embed-A}
		Let $\alm\in\N$.
		For any separable metric space $X$ and non-empty and proper $\om\subset X$, $\B(\om)$ isometrically embeds into $\A(\scut)$ via the map
		\[\sum_{i=1}^k \ldir x_i\rdir \mapsto \sum_{i=1}^k \ldir x_i\rdir + (2m-k)\ldir \partial \rdir.\]
	\end{corollary}

	\begin{proof}
		Embed $\B(X)$ into $\bWass^1_p(X)$ by $\mu\mapsto \mu/m$, apply \Cref{scut-equiv}, and then embed into $\A(\scut)$ by $\mu\mapsto m\mu$.
	\end{proof}
	
	\section{A bi-Lipschitz description of \texorpdfstring{$\A(\scut)$}{Am(Ω*)} in terms of \texorpdfstring{$\A(\R^{n+1})$}{Am(Rn+1)}}\label{coord-proj}

	To construct the bi-Lipschitz embedding from \Cref{main-intro}, it would be natural to adapt the techniques from the proof of \Cref{almgren} to our setting.
	However, the proof of \Cref{almgren} strictly depends on both, the linear structure of $\R^n$ (in particular the existence of projections), and the compactness of the unit ball.
	Although $\om\subset\R^n$ as a set, $\delta$ bears no relationship to the linear structure of $\R^n$ and this fact prohibits the direct use of Almgren's techniques.
	On the other hand, whilst it is possible to find a bi-Lipschitz embedding of $\scut$ into $\ell_2$ to gain a linear structure, this comes at the expense of compactness of the unit ball.
	Thus it is not possible to modify Almgren's proof to our setting.

	In order to prove \Cref{main-intro} we will use a Whitney decomposition $\cubes$ of $\om$ into cubes 
	\[\om = \bigcup_{\cube\in\cubes} \cube\]
	(see \Cref{whitney}) such that, within each $\cube$, $\delta$ is given by $\|\cdot\|$.
	Consequently, $\A(\cube,\delta)=\A(\cube,\|\cdot\|)$.
	\Cref{almgren} then gives a bi-Lipschitz embedding of each $\A(Q,\delta)$ into $\R^N$ and it would be favourable to use these embeddings as ``coordinate projections" to construct a global embedding into Hilbert space.
	Of course, the union of the $\A(\cube)$ does not cover $\A(\scut)$ and therefore we cannot simply define coordinate projections by taking restrictions to each $\cube$.
	Nevertheless, the fact that $\A(Q,\delta)=\A(Q,\|\cdot\|)$ enables us to construct a map $\hLi_\cube\colon \A(\scut)\to \A(\R^{n+1})$ which, roughly speaking, acts as a smooth projection to $\A(\cube)$.

	The main result of this section shows that the $\hLi_\cube$ can be combined to define a bi-Lipschitz embedding of $\A(\scut)$ into the following metric space.
	\begin{definition}\label{lsum-def}
		Let $\cubes$ be a countable set and define
		\begin{equation*}
			\lsum:=\sum_{\cube\in\cubes} \A(\R^{n+1})
		\end{equation*}
		to be the $\ell_2$-sum of copies of $\A(\R^{n+1})$.
		That is, $\lsum$ consists of sequences
		\[\sum_{\cube\in\cubes} a_\cube\]
		of elements of $\A(\R^{n+1})$ for which
		\[\sum_{\cube\in\cubes}\Wass_2^2(a_\cube,0)<\infty,\]
		where $0=\sum_{i=1}^\alm \ldir 0\rdir$,
		equipped with the metric
		\[\sqrt{\sum_{\cube\in\cubes} \Wass_2^2(a_\cube,a'_\cube)}.\]
	\end{definition}

	Once we have an embedding into $\lsum$, we will show that it is possible to find an embedding into $\ell_2$.
	Indeed, in \Cref{sec4}, we apply \Cref{almgren} to each term in the definition of $\lsum$ to obtain a bi-Lipschitz embedding of $\lsum$ into $\ell_2$.

	\subsection{A Whitney decomposition \texorpdfstring{of $\om$}{}}
	
	To construct the embedding into $\lsum$, we will use a Whitney decomposition of $\om$.
	For a cube $\cube\subset \R^n$, let $l(\cube)$ denote the side length of $\cube$.
	\begin{proposition}[Appendix J \cite{zbMATH06313565}]
		\label{whitney}
		Let $\om\subset \R^n$ be non-empty, open and proper.
		There exists a family of closed cubes $\cubes$ such that
		\begin{enumerate}
			\item $\cup \cubes = \om$ and the elements of $\cubes$ have disjoint interiors.
			\item $\sqrt{n} l(\cube) \leq \dist(\cube,\partial\om) \leq 4 \sqrt{n} l(\cube)$ for all $\cube\in \cubes$.
			\item If $\cube,\cube'\in \cubes$ and $\cube\cap \cube'\neq\emptyset$ then
			\[\frac{1}{4} \leq \frac{l(\cube)}{l(\cube')} \leq 4.\]
			We say that $\cube,\cube'$ are \emph{neighbours}.
			\item Each $\cube\in \cubes$ has at most $12^n$ neighbours.
		\end{enumerate}
	\end{proposition}

	A Whitney decomposition of $\om$ estimates which quantity attains the minimum in the definition of $\delta$.
	\begin{lemma}
		\label{whitney-delta}
		Let $\cubes$ be a Whitney decomposition of $\om\subset \R^n$, $\cube,\cube'\in \cubes$ and $x\in \cube$ and $y\in \cube'$.
		Then
		\begin{equation}
			\label{length-dist-boundary}
			\sqrt{n} l(\cube) \leq \dist(x,\partial\om) \leq 5\sqrt{n} l(\cube).
		\end{equation}
		If $\cube,\cube'$ are neighbours then
		\begin{equation}
			\label{nbour}
			\delta(x,y) = \|x-y\|.
		\end{equation}
		If $\cube,\cube'$ are not neighbours then
	\begin{equation}
		\label{not-nbour}
		\frac{l(\cube)+l(\cube')}{8} \leq \delta(x,y) \leq 5\sqrt{n}(l(\cube)+l(\cube')).
	\end{equation}
	\end{lemma}

	\begin{proof}
		The first inequality in \eqref{length-dist-boundary} is implied by $\sqrt{n}l(\cube)\leq \dist(Q,\partial\om)$.
		The second follows from the triangle inequality:
		\begin{equation*}
			\dist(x,\partial\om) \leq \dist(Q,\partial\om) + \operatorname{diam}(Q) \leq 4\sqrt{n}l(\cube) + \sqrt{n}l(\cube).
		\end{equation*}
	Now suppose $\cube,\cube'$ are neighbours and let $z\in Q\cap Q'$.
	Then by \eqref{length-dist-boundary},
	\begin{align*}
		\dist(x,\partial \om) + \dist(y,\partial \om) &\geq \sqrt{n}(l(\cube)+l(\cube'))\\
		&\geq \|x-z\| + \|z-y\| \geq \|x-y\|,
	\end{align*}
	giving \eqref{nbour}.
	On the other hand, suppose that $\cube,\cube'$ are not neighbours and $l(\cube)\geq l(\cube')$.
	Then $\|x-y\|\geq l(\cube'')$ for $\cube''$ a neighbour of $\cube$.
	In particular
	\[\|x-y\|\geq l(\cube'')\geq \frac{l(\cube)}{4} \geq \frac{l(\cube)+l(\cube')}{8},\]
	giving the first inequality in \eqref{not-nbour}.
	The second inequality follows from \eqref{length-dist-boundary}.
\end{proof}

	For the remainder of the paper we fix $\alm\in\N$, $\om\subset \R^n$ non-empty, open and proper and $\cubes$ a Whitney decomposition of $\om$ as in \Cref{whitney}.
	We also fix $\lsum$ as in \Cref{lsum-def}.
	
	\subsection{Constructing a coordinate system}
	To construct a bi-Lipschitz embedding of $\A(\scut)$ into $\lsum$,
	we define projections 
	\[
		\hLi_\cube\colon\A(\scut)\to\A(\R^{n+1})
	\] that serve as a coordinate system for $\A(\scut)$.
	The embedding into $\lsum$ will then be defined as the $\ell_2$-sum of the $\hLi_\cube$ (see \Cref{sum-of-wass}).

	We begin with the construction of a function $\L_\cube$ that approximates the identity within a given $\cube\in\cubes$, is supported on the neighbours of $\cube$, and maintains bi-Lipschitz bounds with $\delta$.
	For $\cube\in\cubes$  and $r>0$, we write $B(\cube,r)$ for the \emph{closed} $r$-neighbourhood of $\cube$.  

	\begin{lemma}
		\label{deltaembedding}
		For each $\cube\in\cubes$ there exists a map
		\[\L_\cube \colon \om \to \R^{n+1}\]
		such that
		\begin{enumerate}
			\item \label{local-lipschitz} $\L_\cube$ is $9\sqrt{n+1}$-Lipschitz;
			\item \label{local-zero} $\L_\cube(x)=0$ for all $x\not\in B(\cube,l(\cube)/4)$.
			In particular, $\L_\cube$ is supported on the neighbours of $\cube$;
			\item \label{local-norm} $\|\L_\cube\|_\infty\leq \sqrt{n+1}l(\cube)$;
			\item \label{local-id} For all $x,y\in B(\cube,l(\cube)/8)$,
			\[\|\L_\cube(x)-\L_\cube(y)\|=\|x-y\|;\]
			\item \label{delta-lip} The extension of $\L_\cube$ to $\scut$, defined by $\L_\cube(\partial)=0$, is $9\sqrt{n+1}$-Lipschitz with respect to $\delta$;
			\item \label{local-biLip} If $x\in B(\cube,l(\cube)/8)$ and $y\in \scut$, then
			\begin{equation*}
				\|\L_\cube(x)-\L_\cube(y)\| \geq \min\left\{\frac{\|x-y\|}{2\sqrt{n}}, l(Q)\right\}.
			\end{equation*}
		\end{enumerate}
	\end{lemma}

	\begin{proof}
		Fix $\cube\in\cubes$ and let $c$ be the centre of $\cube$.
		For each $x\in\om$, let
		\[\eta(x)= \max\left\{1-\dist \left(x,B\left(\cube, \frac{l(\cube)}{8}\right)\right) \frac{8}{l(\cube)},0\right\}.\]
		That is, $\eta$ is an $8/l(\cube)$-Lipschitz function with $\|\eta\|_\infty=1$ that equals 1 on $B(\cube,l(\cube)/8)$ and 0 on $\om\setminus B(\cube,l(\cube)/4)$.
		We also set
		\[\varphi(x)=(x-c,l(\cube)) \in \R^{n+1},\]
		a 1-Lipschitz function satisfying $\|\varphi(x)\|\leq \sqrt{n+1}l(\cube)$ for all $x$ in the support of $\eta$.
		
		Define $\L_\cube = \eta \varphi$.
		Since $\L_\cube$ is a product of Lipschitz functions,
		the Lipschitz constant of $\L_\cube$ is bounded above by
		\[\Lip \varphi \|\eta\|_\infty + \sup\{\|\varphi(x)\|: x\in \operatorname{spt} \eta\} \Lip\eta \leq 1 + \sqrt{n+1}l(\cube) \frac{8}{l(\cube)} \leq 9\sqrt{n+1}.\]
		This demonstrates \cref{local-lipschitz}.
		\Cref{local-zero,local-norm,local-id} are immediate.

		To see \cref{delta-lip}, first let $x\in\scut$ be such that $\L_\cube (x)\neq 0$.
		Then by \cref{local-zero}, $x\in \cube'$ for $\cube'$ a neighbour of $\cube$, so that $l(\cube')\geq l(\cube)/4$.
		Therefore, by \cref{local-norm},
		\begin{align*}
			\|\L_{\cube}(x)\| &\leq 
			\sqrt{n+1}\ l(\cube)\nonumber\\
			&\leq 4\sqrt{n+1}\ l(\cube')\\
			&\leq 8\dist(x,\partial\om),
		\end{align*}
		using \cref{length-dist-boundary} for the final inequality.
		Thus
		\begin{equation*}
			\|\L_\cube(x)\| \leq 8\dist(x,\partial\om)
		\end{equation*}
		holds for any $x\in \scut$ (including $x=\partial$).
		Therefore, by the triangle inequality, for any $x,y\in \scut$,
		\begin{align*}
			\|\L_{\cube}(x)-\L_{\cube}(y)\| \leq 8\left(\dist(x,\partial\om) +\dist(y,\partial\om)\right).
		\end{align*}
		Combining this inequality with \cref{local-lipschitz} shows that $\L_{\cube}$ is $9\sqrt{n+1}$-Lipschitz with respect to $\delta$ on $\scut$.

		Finally, to see \cref{local-biLip}, first suppose that $y\not\in B(\cube,l(\cube)/4)$.
		Then by \cref{local-zero},
		\[\|\L_{\cube}(x)-\L_{\cube}(y)\| = \|\varphi(x)\| \geq l(\cube),\]
		so that \cref{local-biLip} holds in this case.
		
		In the case $y\in B(\cube,l(\cube)/4)$ we will show that
		\begin{equation}
			\|\L_{\cube}(x)-\L_{\cube}(y)\| \geq \frac{\|x-y\|}{2\sqrt{n}}, \label{component-alternate}
		\end{equation}
		completing the proof of \cref{local-biLip}.
		To this end, note that
		\[\|y-c\| \leq \sqrt{n}\frac{l(\cube)}{2} + \frac{l(\cube)}{4} \leq \sqrt{n}l(\cube).\]
		Therefore, by considering the first component of $\L_\cube$, we see that
		\begin{align*}
			\|\L_\cube(x)-\L_\cube(y)\| &\geq 
			\|(x-c) -\eta(y)(y-c)\|\\
			&\geq \|x-y\| -(1-\eta(y))\|y-c\|\\
			&\geq \|x-y\| -\sqrt{n}(1-\eta(y))l(\cube).
		\end{align*}
		Thus, if
		\begin{equation}
			\label{bilip-alt}
			\sqrt{n}(1-\eta(y))l(\cube) \leq \frac{\|x-y\|}{2},
		\end{equation}
		then \eqref{component-alternate} holds.
		On the other hand, if \eqref{bilip-alt} does not hold, then by considering the final component of $\L$, we have
		\[\|\L_\cube(x)-\L_\cube(y)\| \geq (1-\eta(y))l(\cube) \geq \frac{\|x-y\|}{2\sqrt{n}},\]
		giving \eqref{component-alternate}.
	\end{proof}
	
	The pushforwards under each $\phi_\cube$ define our coordinate projections on $\A(\om)$.
	\begin{definition}
		\label{vectembed}
		For every $\cube\in \cubes$, define $\hLi_\cube$ to be the pushforward under $\L_\cube$.
		That is,
		\begin{align*}
			\hLi_\cube \colon \A(\scut) &\to \A(\R^{n+1})\\
			\sum_{i=1}^\alm \ldir p_i \rdir &\mapsto \sum_{i=1}^\alm \ldir \L_\cube(p_i) \rdir.
		\end{align*}
	\end{definition}

	Recall the construction of $\lsum$ from \Cref{lsum-def}.
	\begin{definition}\label{sum-of-wass}
		Define the embedding $\hLi$ by
		\begin{align*}
			\A(\scut) &\to \lsum\\
			\hLi &= \sum_{\cube\in\cubes} \hLi_\cube
		\end{align*}
	\end{definition}

	This is well defined since each $\L_\cube$ is supported on the neighbours of $\cube$, so that each $x\in \om$ is contained in the support of at most $12^n$ of the $\L_\cube$.

	\subsection{\texorpdfstring{$\hLi$ is bi-Lipschitz}{The first embedding is bi-Lipschitz}}
	In this section we show that $\hLi$ is a bi-Lipschitz embedding, beginning by showing that it is Lipschitz.
	
	For $p\in(\scut)^\alm$ and $S\subset \scut$, let
	\[p^{-1}(S)=\{1\leq k\leq \alm : p_k\in S\}.\]
	From now on we use the notation $\sigma q$ to denote the element of $(\R^n)^\alm$ arising from the natural action of the symmetric group $\Sigma_\alm$ on $(\R^n)^\alm$: $(\sigma q)_i = q_{\sigma(i)}$ for each $1\leq i \leq \alm$.

	\begin{lemma}
		\label{vectorbetterLip}
		For any $p,q\in \A(\scut)$,
		\[\sum_{\cube\in\cubes}\Wass_2(\hLi_\cube(p),\hLi_\cube(q))^2 \leq c_0 \Wass_2^2(p,q),\]
		where $c_0\geq 1$ depends only upon $n$.
	\end{lemma}

	\begin{proof}
		Fix $p,q\in (\scut)^\alm$ and let $\cube\in\cubes$ and $\sigma\in\Sigma_m$.
		Set
		\[J_{\cube}^{\sigma}=p^{-1}(B(\cube, l(\cube)/4))\cup (\sigma q)^{-1}(B(\cube, l(\cube)/4)),\]
		so that, by \Cref{deltaembedding} \cref{local-zero},
		\begin{equation*}
			\sum_{k=1}^m\|\L_{\cube}(p_k)-\L_{\cube}(q_{\sigma(k)})\|^2 = \sum_{k\in J_\cube^\sigma}\|\L_{\cube}(p_k)-\L_{\cube}(q_{\sigma(k)})\|^2.
		\end{equation*}
		Applying \Cref{deltaembedding} \cref{delta-lip} gives
		\begin{equation*}
			\sum_{k=1}^m\|\L_{\cube}(p_k)-\L_{\cube}(q_{\sigma(k)})\|^2 \leq 9^2(n+1) \sum_{k\in J_\cube^\sigma} \delta(p_k,q_{\sigma(k)})^2.
		\end{equation*}
		Therefore
		\begin{equation*}
			\sum_{\cube\in\cubes}\min_{\sigma\in\Sigma_\alm}\sum_{k=1}^m\|\L_{\cube}(p_k)-\L_{\cube}(q_{\sigma(k)})\|^2 \leq 9^2(n+1)\sum_{\cube\in\cubes} \min_{\sigma\in\Sigma_\alm}\sum_{k\in J_\cube^\sigma} \delta(p_k,q_{\sigma(k)})^2.
		\end{equation*}
		Further,
		\begin{align*}
			\sum_{\cube\in \cubes}\min_{\sigma\in\Sigma_\alm}\sum_{k\in J_\cube^\sigma} \delta(p_k,q_{\sigma(k)})^2 &\leq \min_{\sigma\in\Sigma_\alm}\sum_{\cube\in\cubes}\sum_{k\in J_\cube^\sigma} \delta(p_k,q_{\sigma(k)})^2\\
			&\leq \min_{\sigma\in\Sigma_\alm} 2\cdot 12^n \sum_{k=1}^\alm \delta(p_k,q_{\sigma(k)})^2,
		\end{align*}
		since $B(\cube,l(\cube)/4)$ is contained within the union of the neighbours of $\cube$.
		The result follows for $c_0= 2\cdot 9^2 \cdot 12^n (n+1)$.
	\end{proof}

	To prove the lower Lipschitz bound, we fix the following notation until the end of the section.

	\begin{notation}
		\label{notation}
		Fix $p,q\in (\scut)^\alm$ and, for every $\cube\in\cubes$, let $\sigma_\cube\in\Sigma_\alm$ be such that 
		\begin{equation}
			\label{min-sigma}
			\sum_{k=1}^m\|\L_{\cube}(p_k)-\L_{\cube}(q_{\sigma_Q(k)})\|^2 = \Wass_2(\hLi_\cube(p),\hLi_\cube(q))^2.
		\end{equation}
		
		Let $\cube \in \cubes$.
		For integer $0\leq r \leq 2m$, the annuli
		\[\cube^r=B\left(\cube,\frac{r+1}{3\alm} \frac{l(\cube)}{8}\right)\setminus B\left(\cube,\frac{r}{3\alm} \frac{l(\cube)}{8}\right)\]
		are disjoint and so there exists $0\leq r \leq 2\alm$ such that
		\begin{equation}
			\label{empty-annulus}
			p^{-1}(\cube^{r})\cup (\sigma_\cube q)^{-1}(\cube^{r})=\emptyset.
		\end{equation}
		Set
		\begin{equation*}
			\widehat \cube= B\left(\cube,\frac{r}{3\alm} \frac{l(\cube)}{8}\right).
		\end{equation*}
		Note that $\widehat\cube$ is contained within the union of the neighbours of $\cube$.
		
		Let $c_1=(48\sqrt{n})^{-1}$ and define $\cubes'$ to be the set of $\cube\in\cubes$ for which
		\begin{equation}\label{close-match}
			\Wass_2(\hLi_\cube(p),\hLi_\cube(q)) < c_1 \frac{l(\cube)}{\alm}.
		\end{equation}
		Set
		\begin{equation*}
			E = \bigcup_{\cube\in\cubes'} \widehat\cube.
		\end{equation*}
		
	\end{notation}

	To obtain a lower bound of
	\begin{equation}\label{sum-cube-cost}
		\sum_{\cube\in\cubes}\Wass_2(\hLi_\cube(p),\hLi_\cube(q))^2
	\end{equation}
	in terms of $\Wass_2^2(p,q)$, we will construct a $\tau\in\Sigma_m$ for which $\sum_{i=1}^m\delta(p_i,q_{\tau(i)})^2$ is comparable to \eqref{sum-cube-cost}.
	A first attempt to do this may be, for each $\cube\in\cubes$ and each $i\in p^{-1}(\cube)$, to define $\tau(i)=\sigma_\cube(i)$.
	Of course, a $\tau$ defined in this way need not be injective, for example if there exist $\cube\neq\cube'\in \cubes$ and $i\neq j$ such that $q_{\sigma_\cube(i)} = q_{\sigma_{\cube'}(j)}$.
	Nonetheless, we will show that it is possible to construct a permutation for the cubes in $\cubes'$.
	Indeed, we now show that conditions \eqref{empty-annulus} and \eqref{close-match} ensure that, for each $\cube\in\cubes'$, $p_i\in\widehat{\cube}$ if and only if $q_{\sigma_\cube(i)}\in \widehat \cube$:
	\eqref{empty-annulus} provides a moat surrounding $\widehat \cube$ and \eqref{close-match} ensures that the distance between $p_i$ and $q_{\sigma_\cube(i)}$ is less than the width of the moat.

	\begin{lemma}
		\label{annbij}
		For any $\cube\in\cubes'$,
		\begin{equation}\label{eq-bij}
			p^{-1}(\widehat\cube) = (\sigma_\cube q)^{-1}(\widehat \cube)
		\end{equation}
		and
		\begin{equation}
			\label{inner-isom}\|p_k-q_{\sigma_\cube(k)}\| =\|\L_\cube(p_k)-\L_\cube(q_{\sigma_\cube(k)})\| \quad \forall k\in p^{-1}(\widehat\cube).
		\end{equation}
		Moreover, if $R\in\cubes'$ with $l(R)\leq l(\cube)$,
		\begin{equation} \label{equal-index}
			p^{-1}(\widehat\cube \cap \widehat R)= (\sigma_R q)^{-1}(\widehat\cube \cap \widehat R)
		\end{equation}
	\end{lemma}

	\begin{proof}
		For any $k\in p^{-1}(\widehat\cube)$, \eqref{close-match} and \Cref{deltaembedding} \cref{local-biLip} imply
		\[
		\min\left\{\frac{\|p_k-q_{\sigma_\cube(k)}\|}{2\sqrt{n}},l(\cube)\right\} < c_1\frac{l(\cube)}{\alm}.	
		\]
		In particular,
		\begin{equation}
			\label{cant-jump-moat}
			\|p_k-q_{\sigma_\cube(k)}\| < \frac{l(\cube)}{24\alm}.
		\end{equation}
		Therefore \eqref{empty-annulus} implies that $q_{\sigma_\cube(k)}\in \widehat\cube$.
		By symmetry, if $k\in (\sigma_\cube q)^{-1}(\widehat \cube)$ then $k\in p^{-1}(\widehat \cube)$ and so \eqref{eq-bij} holds.
		Since $\widehat\cube\subset B(\cube,l(\cube)/8)$, \Cref{deltaembedding} \cref{local-id} implies \eqref{inner-isom}.

		Now let $R\in\cubes'$ with $l(R)\leq l(\cube)$ and $k\in p^{-1}(\widehat\cube \cap \widehat R)$.
		Then \eqref{cant-jump-moat} for $R$ implies
		\begin{equation*}
			\|p_k-q_{\sigma_R(k)}\| < \frac{l(R)}{24m} \leq \frac{l(\cube)}{24m}
		\end{equation*}
		and so \eqref{empty-annulus} implies $q_{\sigma_R(k)}\in \widehat \cube$.
		The similar argument with $p$ and $\sigma_R q$ exchanged gives \eqref{equal-index}.
	\end{proof}

	By carefully partitioning $E$ using the $\widehat Q$, we use \Cref{annbij} to construct the desired permutation on $p^{-1}(E)$.
	\begin{proposition}
		\label{lowerbound-euclidean}
		There exists a bijection $\tau\colon p^{-1}(E)\to q^{-1}(E)$ such that
		\begin{equation*}
			\sum_{k\in p^{-1}(E)}\|p_k-q_{\tau(k)}\|^2 \leq \sum_{\cube\in\cubes'}\Wass_2(\hLi_\cube(p),\hLi_\cube(q))^2.
		\end{equation*}
	\end{proposition}

	\begin{proof}
	Let
	\[\cubes''=\{\cube\in\cubes' : p^{-1}(\widehat\cube) \neq\emptyset\}.\]
	Note that, by \eqref{eq-bij}, $\cubes''$ can equivalently be defined as the set of $\cube\in\cubes$ with $q^{-1}(\widehat\cube) \neq\emptyset$.
	Since $\cubes''$ is finite, we enumerate it as
		\begin{equation*}
			\cubes'' = \{\cube_1,\cube_2,\ldots,\cube_j\}
		\end{equation*}
		in such a way that
		\begin{equation*}
			l(\cube_1) \geq l(\cube_2) \geq \ldots \geq l(\cube_j).
		\end{equation*}
		Then, for $1\leq i\leq k\leq j$, applying \Cref{annbij} with $\cube=\cube_k$ and $R=\cube_i$ gives
		\begin{equation}
			\label{claim-step}
			p^{-1}\left(\widehat\cube_i \cap \widehat\cube_k\right)= (\sigma_{\cube_k} q)^{-1}\left(\widehat\cube_i \cap \widehat\cube_k\right) \quad \forall 1\leq i \leq k \leq j.
		\end{equation}

		Let $B_1 = \widehat \cube_1$ and for each $2\leq k\leq j$ define
		\[B_k:= \widehat \cube_k \setminus \bigcup_{i=1}^{k-1} \widehat \cube_i = \widehat \cube_k \setminus \bigcup_{i=1}^{k-1} \widehat \cube_i \cap \widehat \cube_k.\]
		Then \eqref{claim-step} implies that $\sigma_{\cube_k}$ is a permutation between $p^{-1}(B_k)$ and $\sigma_{\cube_k}q^{-1}(B_k)$ for each $1\leq k\leq j$.
		Therefore, we define a bijection 
		\begin{equation*}
			\tau \colon p^{-1}(E) \to q^{-1}(E)
		\end{equation*}
		by setting $\tau$ to equal $\sigma_{\cube_k}$ on $D_k:=p^{-1}(B_k)$ for each $1\leq k\leq j$.
		Then
		\begin{align*}
			\sum_{k\in p^{-1}(E)} \|p_k-q_{\tau(k)}\|^2 &= \sum_{i=1}^j\sum_{k\in D_i} \|p_k-q_{\tau(k)}\|^2 \\
			&= \sum_{i=1}^j\sum_{k\in D_i} \|p_k-q_{\sigma_{\cube_i}(k)}\|^2 \\
			&= \sum_{i=1}^j\sum_{k\in D_i} \|\L_{\cube_i}(p_k)-\L_{
				\cube_i}(q_{\sigma_{\cube_i}(k)})\|^2 \\
			&\leq \sum_{\cube\in\cubes''}\sum_{k=1}^m \|\L_{\cube}(p_k)-\L_{
				\cube}(q_{\sigma_{\cube}(k)})\|^2, 
		\end{align*}
		using \eqref{inner-isom} for the third equality.
		Finally \eqref{min-sigma} completes the proof.
	\end{proof}

	Next we consider the points outside $E$ for which we use the distance to $\partial \om$ to estimate $\delta$.
	\begin{lemma}
		\label{lowerbound-boundary}
		For any bijection
		\begin{equation*}
			\sigma\colon p^{-1}(\om\setminus E) \to q^{-1}(\om\setminus E)
		\end{equation*}
		we have
		\begin{equation*}
			\sum_{k\in p^{-1}(\om\setminus E)} (\dist(p_k,\partial\om) + \dist(q_{\sigma(k)},\partial\om))^2 \leq m^3 c_2 \sum_{\cube\in\cubes\setminus \cubes'} \Wass_2(\hLi_\cube(p),\hLi_\cube(q))^2,
		\end{equation*}
		for $c_2\geq 1$ that depends only upon $n$.
	\end{lemma}

	\begin{proof}
		For a moment fix $k\in p^{-1}(\om\setminus E)$ and let $\cube\in\cubes$ contain $p_k$.
		Then necessarily $Q\not\in \cubes'$.
		Therefore \eqref{close-match} and \eqref{length-dist-boundary} imply
		\begin{equation*}
		\Wass_2(\hLi_\cube(p),\hLi_\cube(q)) \geq \frac{c_1}{m}l(\cube) \geq \frac{c_1}{5\sqrt{n}m} \dist(p_k,\partial \om).
		\end{equation*}
		Since each $\cube\in\cubes$ contains at most $m$ such points $p_k$,
		\begin{align*}
			\sum_{k\in p^{-1}(\om\setminus E)} \dist(p_k,\partial\om)^2 &\leq \frac{25 m^2 n}{c_1^2} m \sum_{\cube\in\cubes\setminus \cubes'} \Wass_2(\hLi_\cube(p),\hLi_\cube(q))^2.
		\end{align*}
		The same estimate for $\sigma q$ gives the desired inequality for $c_2=4\cdot 25 n/c_1^2$.
	\end{proof}

	We combine our previous results to show that $\hLi_\cube$ is a bi-Lipschitz embedding.
	\begin{theorem}\label{embed-into-wass}
	    For any $p,q\in\A(\scut)$,
	    \begin{equation*}
		    \frac{\Wass_2(p,q)^2}{c_3 \alm^3} \leq \sum_{\cube\in\cubes} \Wass_2(\hLi_\cube(p),\hLi_\cube(q))^2 \leq c_3\Wass_2(p,q)^2,
	    \end{equation*}
		where $c_3\geq 1$ depends only upon $n$.
	\end{theorem}

	\begin{proof}
		The right hand inequality is given by \Cref{vectorbetterLip}.

		For the left hand inequality, let $\tau$ be the bijection obtained from \Cref{lowerbound-euclidean} and arbitrarily extend it to a bijection of $\{1,\ldots,\alm\}$.
		Then
		\begin{align*}
			\sum_{\cube\in\cubes} \Wass_2(\hLi_\cube(p),\hLi_\cube(q))^2 &= \sum_{\cube\in\cubes'}\sum_{k=1}^m \|\L_{\cube}(p_k)-\L_{
				\cube}(q_{\sigma_{\cube}(k)})\|^2 \\
			&\quad + \sum_{\cube\not\in\cubes'}\sum_{k=1}^m \|\L_{\cube}(p_k)-\L_{
				\cube}(q_{\sigma_{\cube}(k)})\|^2 \\
			&\geq \sum_{k\in p^{-1}(E)} \|p_k-q_{\tau(k)}\|^2\\
			&\quad + \frac{1}{c_2 \alm^3}\sum_{k\in p^{-1}(\om\setminus E)} (\dist(p_k,\partial\om) + \dist(q_{\tau(k)},\partial\om))^2\\
			&\geq \frac{1}{c_2 \alm^3} \sum_{k=1}^{\alm}\delta(p_k,q_{\tau(k)})^2\\
			&\geq \frac{1}{c_2 \alm^3} \Wass_2(p,q)^2,
		\end{align*}
	using \Cref{lowerbound-euclidean} and \Cref{lowerbound-boundary} for the first inequality.
	\end{proof}

	\section{The embedding into Hilbert space}\label{sec4}

	In this section we conclude the proof of \Cref{main-intro}.
	Let $\xi \colon \A(\R^{n+1}) \to \R^N$ be the embedding given by \Cref{almgren}.
	We write
	\[\ell_2 = \sum_{\cube\in \cubes} \R^N\]
		as a direct $l_2$-sum over $\cubes$.
	Recall the construction of $\lsum$ from \Cref{lsum-def}.
	\begin{lemma}\label{embedding}
		The function $\xi'\colon \lsum\to\ell_2$ defined by
		\begin{align*}
			\sum_{\cube\in\cubes} \A (\R^{n+1})&\to \sum_{\cube\in\cubes} \R^N \\
			\xi' = \sum_{\cube\in\cubes} \xi 
		\end{align*}
		is well defined.
		Moreover, for any $a,b\in \lsum$,
		\begin{equation*}
			\frac{1}{c \alm^{2n+2}} \sum_{\cube\in\cubes} \Wass_2(a_\cube,b_\cube)^2 \leq \|\xi'(a)-\xi'(b)\|^2 \leq \sum_{\cube\in\cubes} \Wass_2(a_\cube,b_\cube)^2,
		\end{equation*}
		for $c\geq 1$ depending only upon $n$.
		
	\end{lemma}
	\begin{proof}
	Let $a\in\lsum$, so that
	\[
	\sum_{\cube \in\cubes} \Wass_2(p_\cube,0)^2 <\infty.
	\]
	Since $\xi$ is 1-Lipschitz this implies that
	\[
		\sum_{\cube \in\cubes} \|\xi(p_\cube)\|^2 = \sum_{\cube \in\cubes} \|\xi(p_\cube)-\xi(0)\|^2 \leq \sum_{\cube \in\cubes} \Wass_2(p_\cube,0)^2 <\infty. 
	\]
	Hence, $\xi'$ is well defined.
	Moreover, using that $\xi$ is 1-Lipschitz again, we have, for any $b\in\lsum$,
	\[
		\sum_{\cube \in\cubes} \|\xi(a_\cube)-\xi(b_\cube)\|^2 \leq \sum_{\cube \in\cubes} \Wass_2(a_\cube,b_\cube)^2,
	\]
	so that $\xi'$ is also 1-Lipschitz.
	Finally, \Cref{almgren} gives
	\[
		\sum_{\cube \in\cubes} \|\xi(a_\cube)-\xi(b_\cube)\|^2 \geq \frac{1}{cm^{2n+2}}\sum_{\cube \in\cubes} \Wass_2(a_\cube,b_\cube)^2.
	\]
	\end{proof}

	\begin{theorem}
		\label{main}
		There exists a bi-Lipschitz embedding $\zeta \colon \B(\om)\to \ell_2$ with distortion at most $cm^{n+5/2}$, for $c\geq 1$ depending only upon $n$.
		That is, for any $p,q\in \B(\om)$,
		\begin{equation*}
			\frac{\Wass_2(p,q)}{c \alm^{n+5/2}} \leq \|\zeta(p)-\zeta(q)\| \leq c \Wass_2(p,q).
		\end{equation*}
	\end{theorem}

	\begin{proof}
		First isometrically embed $\B(\om)$ into $\A(\scut)$ via \Cref{B-embed-A}.
		One then applies \Cref{embed-into-wass} to bi-Lipschitz embed $\A(\scut)$ into $\lsum$.
		Finally, \Cref{embedding} bi-Lipschitz embeds $\lsum$ into $\ell_2$, as required.
	\end{proof}

	\begin{remark}\label{distortion}
		For $n\geq 3$, the distortion of any embedding of $\A(\scut)$ into $\ell_2$ converges to $\infty$ as $\alm$ increases.
		In particular, $\bWass_2(\om)$ does not bi-Lipschitz embed into $\ell_2$.

		Indeed, by \Cref{nbour} we see that $\A(\scut)$ contains an isometric copy of $\A(\cube)$ for some cube $\cube$. 
		Thus, the distortion of any embedding into $\ell_2$ is at least that of $\A(\cube)$.
		For $n\geq 3$, Andoni, Naor and Nieman \cite[Theorem 7]{naor} prove that $\Wass_2(\R^n)$ does not coarsely, in particular bi-Lipschitz, embed into any Banach space of non-trivial type, namely Hilbert space.
		Since the set of discrete measures is dense in $\Wass_2(\R^n)$, a scaling argument shows that the distortion of any bi-Lipschitz embedding of $\A(\cube)$ must converge to $\infty$ as $\alm$ does.

		The same conclusion can be made for $n=2$ using an unpublished result of Austin and Naor announced in \cite[Remark 8]{naor}, which states that $\Wass_2(\R^2)$ does not bi-Lipschitz embed into $L_1$ and, hence, does not bi-Lipschitz embed into $\ell_2$.
	\end{remark}

	\section{Conflict of interests}
	On behalf of all authors, the corresponding author states that there are no conflicts of interest.


\end{document}